\documentclass[11pt]{article}

\usepackage{amsthm,amsmath,amssymb}
\usepackage{verbatim}

\theoremstyle{plain}
\newtheorem{theorem}{Theorem}[section]
\newtheorem{proposition}[theorem]{Proposition}
\newtheorem{lemma}[theorem]{Lemma}

\theoremstyle{definition}

\newcommand{\HS}{Hajnal-Szemer\'{e}di}
\newcommand{\Z}{\mathbb{Z}}
\newcommand{\lilOh}[1]{o(#1)}
\newcommand{\paren}[1]{\left( #1 \right)}
\newcommand{\set}[1]{\left\{ #1 \right\}}

\newtheorem{obs}{Observation}
\newtheorem{conj}{Conjecture}
\newtheorem{ques}{Question}
\newtheorem{fact}{Fact}

\date{}
\title{Towards a weighted version of the \HS\ Theorem}
\author
{J\'ozsef Balogh \thanks{Department of Mathematical Sciences, University of
Illinois, Urbana, IL, 61801, Supported in part by  NSF CAREER Grant DMS-0745185, UIUC Campus Research Board Grants 09072 and 11067, OTKA Grant K76099, and TAMOP-4.2.1/B-09/1/KONV-2010-0005 project.   Work was partly done while at Department of Mathematics, University of California,
San Diego La Jolla, CA, 92093.}
\and
Graeme Kemkes \thanks{Department of Mathematics, Ryerson University
Toronto, ON, M5B 2K3}
\and
Choongbum Lee \thanks{Department of Mathematics, UCLA, Los
Angeles, CA, 90095. Email: choongbum.lee@gmail.com. Research
supported in part by a Samsung Scholarship.}
\and
Stephen J. Young \thanks{Department of Mathematics, University of California,
San Diego, La Jolla, CA, 92093}
}

\begin{document}

\maketitle

\begin{abstract}
For a positive integer $r \ge 2$, a $K_r$-factor of a graph is a collection
of vertex-disjoint copies of $K_r$ which covers all the vertices of the
given graph.
The celebrated theorem of Hajnal and Szemer\'edi asserts that
every graph on $n$ vertices with minimum degree at least $(1-\frac{1}{r})n$ 
contains a $K_r$-factor. In this note, we propose
investigating the relation between minimum degree and existence
of perfect $K_r$-packing for edge-weighted graphs.
The main question we study is the following. Suppose that
a positive integer $r \ge 2$ and a real $t \in [0,1]$ is given.
What is the minimum weighted degree of $K_n$ that guarantees the
existence of a $K_r$-factor such that every factor has total edge
weight at least $t\binom{r}{2}$? We provide some lower and upper bounds
and make a conjecture on the asymptotics of the threshold as
$n$ goes to infinity.  This is the long version of a ``problem paper''
in \emph{Combinatorics, Probability and Computing}.
\end{abstract}

\section{Introduction}
\label{firstpage}

Many results in graph theory study the relation between the
minimum degree of a given graph and its spanning subgraphs. For example,
Dirac's theorem asserts that a graph on $n$ vertices with minimum degree
at least $\lceil \frac{n}{2} \rceil$ contains a Hamilton cycle. 
Hajnal and Szemer\'edi \cite{Hajnal:EquitablePartition}
proved that every graph on $n \in r\Z$ vertices
with minimum degree at least $(1-\frac{1}{r})n$ contains a
spanning subgraph consisting of $\frac{n}{r}$ vertex-disjoint copies of $K_r$
(we call such a subgraph a $K_r$-factor). 

In this note 
we  propose investigating this relation in edge-weighted graphs. 
As a concrete problem, we study
the particular case when the spanning subgraph is the graph formed
by vertex-disjoint copies of $K_r$ (in other words, we would like to
extend the Hajnal-Szemer\'edi theorem to edge-weighted graphs).
Suppose we equip the complete graph $K_n$ with edge weights
$w \colon E(K_n) \rightarrow [0,1]$. 
For a given weighted graph and vertex $v$ we let $\deg_w(v)$ 
denote the weighted degree of the vertex $v$. Let $\delta_w(G)$ be 
the minimum weighted  degree of the graph $G$. 
The main question can be formulated as the following: 
How large must $\delta_w(K_n)$ be to guarantee
that there exists a $K_r$-factor such that every factor has total edge
weight at least $t\binom{r}{2}$ for some given $t \in [0,1]$?

More formally, for $n \in r\Z$ let $\mathcal{W}(r,t,n)$ be the
collection of edge weightings on $K_{n}$ such that every $K_r$-factor
has a clique with weight strictly smaller than $t\binom{r}{2}$.  We
then define \[ \delta(r,t,n) =  \sup_{w
  \in \mathcal{W}(r,t,n)} \delta_w(K_{n}) \ \ \textrm{ and }
\ \ \delta(r,t) = \limsup_{n \rightarrow \infty} \frac{\delta(r,t,n)}{n}.\]

\noindent The main open question that we raise is the following.
\begin{ques}
Determine the value of $\delta(r,t)$ for all $r$ and $t$.
\end{ques}

Let $\mathcal{W}^*(r,t,n)$ be the collection of edge weightings of $K_n$ such
that every $K_r$-factor has a clique with weight at most
$t\binom{r}{2}$ (instead of strictly smaller than $t\binom{r}{2}$), and 
define the functions $\delta^*(r,t,n)$ and $\delta^*(r,t)$ accordingly.

\begin{proposition} \label{prop:continuity}
For all $r,t$, and $n$, $\delta(r,t,n) = \delta^*(r,t,n)$. Therefore,
$\delta(r,t) = \delta^*(r,t)$.
\end{proposition}
\begin{proof}
  The inequality $\delta(r,t,n) \le \delta^*(r,t,n)$ easily follows
  from the definition. Noting that the complement of
  $\mathcal{W}^{*}(r,t,n)$ is open in the set of all real-valued edge
  weightings, the set $\mathcal{W}^{*}(r,t,n)$ is compact.  Thus there
  is a weight function $w \in \mathcal{W}^{*}(r,t,n)$ so that
  $\delta_w(K_n) = \delta^{*}(r,t,n)$. Let $\varepsilon < 1$ be an
  arbitrary positive real, and let $w'$ be the weight function
  obtained from $w$ by multiplying $1-\varepsilon$ to all the weights.
  One can easily see that $w' \in \mathcal{W}(r,t,n)$, and thus
  $\delta(r,t,n) \ge (1-\varepsilon)\delta^*(r,t,n)$.  Thus as
  $\varepsilon$ tends to 0, we see that $\delta(r,t,n) \ge
  \delta^*(r,t,n)$.  This concludes the proof.
\end{proof}

The proposition above shows that if an edge-weighting of $K_n$ 
has minimum degree
greater than $\delta(r,t,n)$, then there exists a $K_r$-factor
such that every copy of $K_r$ has weight greater than $t\binom{r}{2}$.
Therefore, the \HS\ theorem in fact is a special case 
of our problem when $t = (\binom{r}{2} - 1)/\binom{r}{2}$ 
where we only consider the  integer weights $\{0,1\}$. 
Thus we believe that the following special case is an important and 
interesting instance of the problem corresponding to the
\HS\ theorem for $r=3$ (which has been first proved by Corr\'adi and Hajnal \cite{Corradi:IndepCircuits}).
\begin{ques}
What is the value of $\delta(3,\frac{2}{3})$?
\end{ques}

In the rest of our note we describe our partial results toward answering Question 1.

\section{Lower bound} \label{sec:lower}

It is not too difficult to deduce the bound $\delta(r,t) \ge (1-1/r)t$ from
the  graph showing the sharpness of the  Hajnal-Szemer\'edi theorem. 
Our first proposition provides a better lower bound to this function.

\begin{proposition}\label{P:lb}
The following holds for every integer $r \ge 2$ and real $t \in (0,1]$:
\[ \delta(r,t) \geq \frac{1}{r} + \left(1-\frac{1}{r}\right)t. \]
\end{proposition}
\begin{proof}
  Let $n \in r\Z$ with $n > r$ and let $k = \frac{n}{r}$.
  Let $A$ be an arbitrary set of $k-1$ vertices and let
  $B$ be the remaining $k(r-1)+1$ vertices. 
  Consider the weight function $w$ that assigns weight
  $t$ to edges whose  endpoints are both in $B$ and weight 1 to all other
  edges.  By the cardinality of $A$, we see that every $K_r$-factor must
  contain a clique that lies entirely within $B$. Since our weight function
  gives weight at most $t\binom{r}{2}$ to this clique, we see that $w \in
  \mathcal{W}^*(r,t,n)$.  Further, $\delta_{w}(K_n) =
  \min\{n-1, k-1 + t(n-k)\}$.  But we have that 
\[ k-1 + t(n-k) = \left(\frac{1}{r} - \frac{1}{n} +
  t(1-\frac{1}{r})\right)n.\]  
  Therefore by Proposition \ref{prop:continuity}, we have
  $\delta(r,t,n) \ge \left(\frac{1}{r} - \frac{1}{n} +
  t(1-\frac{1}{r})\right)n$ and $\delta(r,t) \ge \frac{1}{r} + \left(1-\frac{1}{r}\right)t$.
\end{proof}

Proposition \ref{P:lb} illustrates the fundamental difference between 
the minimum degree threshold for containing a $K_r$-factor in
graphs and edge-weighted graphs. For example, when $r=3$, we see that
$\delta(3,2/3) \ge 7/9$, while the corresponding function 
for graphs has value $2/3$ by the Hajnal-Szemer\'edi theorem.
This difference suggests that we indeed need some
new ideas and techniques to solve our problem.

\section{Upper bound}

Next, we establish an upper bound on $\delta(r,t)$.
To do so, it is helpful to consider the graph induced 
by the edges of heavy weights in a given edge-weighted graph.
Thus, given an edge-weighted graph $G_w$, 
we denote by $G_w(t)$ the subgraph of $K_n$ consisting of edges of
weight at least $t$. For $r=2$, it is easy to establish the
correct value of the function $\delta(2,t)$.

\begin{obs} \label{obs:simplecase}
 For every $t \in (0,1]$ we have $\delta(2,t) = \frac{1+t}{2}$.
\end{obs}
\begin{proof}
The lower bound on $\delta(2,t)$ follows from Proposition \ref{P:lb}, and thus it 
suffices to establish the upper bound.
Let $w$ be a weight function such
that $\delta_w(K_n) \geq \frac{1+t}{2}n$.  Now
for any vertex $v \in G_w(t)$, we have $\deg_w(v) < (n-1-\deg(v))
\cdot t +
\deg(v)\cdot 1$, where $\deg(v)$ is the degree of $v$ in $G_w(t)$.  But then the minimum weighted degree condition implies
that $\deg(v) \geq \frac{n}{2}$, and so by the  \HS\ theorem there is a 
$K_2$-factor in $G_w(t)$. By the definition of $G_w(t)$, this establishes
the bound $\delta(2,t) \le \frac{1+t}{2}$.
\end{proof}

Even for $r \ge 3$, if $t$ is small enough, then we can determine the correct value of the function $\delta(r,t)$.

\begin{theorem} \label{thm:upper}
For every $r \ge 3$, there exists a positive real $t_r$ such that 
for every $t \in (0,t_r)$ we have 
\[ \delta(r,t) = \frac{1}{r} + \left(1 - \frac{1}{r}\right)t. \]
\end{theorem}
\begin{proof}
We have $\delta(r,t) \geq \frac{1}{r} + \left(1 - \frac{1}{r}\right)t$ by Proposition \ref{P:lb}. It 
remains for us to establish the upper bound. Let $\varepsilon$ be an arbitrary positive real. For $n$ sufficiently large, let $w$ be a weight function such
that $\delta_w(K_n) \geq \left(\frac{1}{r} + \left(1 - \frac{1}{r}\right)t + \varepsilon\right)n$.
We will say a copy of a $K_r$
is {\em heavy} if it has weight at least ${r \choose 2} t$. A
collection of vertex-disjoint copies of $K_r$ is {\em heavy} if
each $K_r$ in the collection is heavy. An edge is
{\em overweight} if it has weight at least ${r \choose 2}t$.
Let $t_r$ be a sufficiently small positive real 
depending on $r$ to be determined later.
We will find a heavy $K_r$-factor given that $t < t_r$ and
$n$ is a large enough integer divisible by $r$.

Take a maximum heavy collection of vertex-disjoint copies of $K_r$, that
maximizes the number of overweight edges. 
Call this collection $\mathcal{R}$, and suppose that $|\mathcal{R}|=\rho$.
Denote by $V_R$ be the vertices covered by $\mathcal{R}$, thus
$|V_R| = r\rho$. We may assume that $\rho < \frac{n}{r}$, as otherwise
we have a heavy $K_r$-factor.
Then there exist $r$ distinct vertices 
$v_1, v_2, \cdots, v_r \notin V_R$. Let $L = \{v_1,
v_2, \cdots, v_r\}$. If there is an overweight edge whose 
endpoints are both in $V(K_n)\setminus V_R$, then we can find
a larger collection than $\mathcal{R}$ by taking the union of this edge 
with $r-2$ vertices of $L$.
Thus all the edges within $V(K_n)\setminus V_R$ have weight 
at most ${r \choose 2} t$.

\begin{fact} \label{fact:fact1}
For every $R \in \mathcal{R}$, if there exists an overweight edge
between $V(R)$ and $L$, then there exists a unique vertex
in $R$ which intersects every overweight edge between $V(R)$ and $L$.
\end{fact}
\begin{proof}
Fix a copy of $K_r$ in $\mathcal{R}$ and denote it by $R$.
If there are two vertex-disjoint overweight edges between
$V(R)$ and $L$, then we can find two heavy
vertex-disjoint copies of $K_r$ over $V(R) \cup L$.
Therefore all the overweight edges between $V(R)$ and $L$
share a common endpoint. In particular, there are at 
most $r$ overweight edges between $V(R)$ and $L$.

Now suppose that there are at least two overweight edges between $V(R)$
and $L$, and that the common endpoint is in $L$. Without loss of generality,
let $x \in V(R)$ and $v_1 \in L$ be vertices such that there are at least 
two overweight edges of the form $\{y,v_1\}$ for $y \in V(R) \setminus \{x\}$.
Then by the assumption that we maximized the number of overweight edges,
there are at least two overweight edges among the edges $\{y, x\}$
for $y \in V(R) \setminus \{x\}$ (otherwise we can replace 
$R$ by $R \setminus \{x\} \cup \{v_1\}$). However, 
if this is the case, then we can find two independent overweight edges
in $V(R) \cup L$, and this contradicts
the maximality of $\mathcal{R}$. Thus if there are at least two overweight
edges between $V(R)$ and $L$, then they share a common endpoint in $V(R)$.
\end{proof}

Let $\mathcal{R}'$ be the subset of copies of $K_r$ of $\mathcal{R}$ 
which have at least $r-1$ overweight edges incident to it 
whose other endpoint is in $L$. Let $\rho' = |\mathcal{R}'|$. 

\begin{fact} \label{fact:fact15}
There exists a real $t_r'$ such that if $t < t_r'$, then 
$\rho' \geq 1$.
\end{fact}
\begin{proof}
By Fact \ref{fact:fact1} we have
\begin{align*} 
  (1 + (r-1)t)n &\le \sum_{i=1}^{r} \deg_w(v_i) \le 
  \left(r + (r^2-r){r \choose 2}t\right)\rho' \\
 & \qquad + \left( r-2 + (r^2-r+2) {r \choose 2}t\right) (\rho-\rho') + r {r \choose 2}t  (n - r\rho).
\end{align*}
By way of contradiction, let $\rho' = 0$ and thus,
\[ (1+ (r-1)t)n \leq (r-2)\left(1-\binom{r}{2}t\right)\rho + r\binom{r}{2}t n.
\]
Now for sufficiently small $t$, the coefficient of $\rho$ is positive
so we may substitute $\rho \leq \frac{n}{r}$ and obtain,
\[ 0 \leq \left( \frac{r-2}{r}\left(1-\binom{r}{2}t\right) +
  r\binom{r}{2}t -1 -(r-1)t\right)n =\left( (r-1)\binom{r}{2}t - \frac{2}{r} \right)n \]
If $t$ is sufficiently small this is a contradiction, and thus $\rho'
\geq 1$.
\end{proof}

\begin{fact} \label{fact:fact2}
For $R \in \mathcal{R}'$, there exists a unique vertex $x_R \in V(R)$
incident to all the overweight edges within $V(R) \cup L$. Moreover, 
all the edges incident to $x_R$ within $R$ are overweight.
\end{fact}
\begin{proof}
For a fixed copy $R \in \mathcal{R}'$, let $x_R \in V(R)$ be the vertex
guaranteed by Fact \ref{fact:fact1}.
If there is an overweight edge in $R$ which
does not intersect $x_R$, then we can find two
heavy vertex-disjoint copies of $K_r$ over the set of vertices $V(R) \cup L$, 
and this violates the maximality of $\mathcal{R}$. Therefore, 
all the overweight edges within $R$ are incident to $x_R$. Moreover, if
there are less than $r-1$ such edges, then we can find a copy of $K_R$ 
over the vertex set $\{x_R\} \cup L$ which contains at least $r-1$ overweight
edges. This contradicts the maximality of overweight edges of
$\mathcal{R}$. Therefore, all the edges incident to $x_R$ within
$R$ are overweight.
\end{proof}

Let $X$ be the subset of
vertices which are covered by copies of $K_r$ in $\mathcal{R}'$ that are
incident to an overweight edge (guaranteed by Fact \ref{fact:fact2}),
and let $Y$ be the vertices which are covered by copies of $K_r$ in
$\mathcal{R}'$ that are not in $X$.

\begin{fact} \label{fact:fact3}
For every $y \in Y$ and $R \in \mathcal{R} \setminus \mathcal{R}'$, $y$
is incident to $R$ by at most one overweight edge.
\end{fact}
\begin{proof}
Suppose that we are given vertices $x \in X$ and $y \in Y$ covered
by $R \in \mathcal{R}$. Without loss of generality, suppose that
$x$ is adjacent to $v_1, \cdots, v_{r-1} \in L$ by overweight edges.
By way of contradiction, suppose that there exists 
$R' \in \mathcal{R} \setminus \mathcal{R}'$ with $V(R') =
\{z_1, z_2, \cdots, z_r\}$ such that $\{y, z_1\}$ and $\{y, z_2\}$ are
both overweight. If $R'$ contains an edge $e$ other than $\{z_1, z_2\}$
that is overweight, then among the edges 
$\{x,v_1\}, \{y,z_1\}, \{y,z_2\}, e$ (which are all overweight), we 
can find at least three vertex-disjoint edges. Therefore we can find
three vertex-disjoint copies of $K_r$ over the vertex set
$V(R) \cup V(R') \cup L$. However, this contradicts the maximality of $\mathcal{R}$.
If there are no overweight edges within $R'$ other than (possibly) $\{z_1, z_2\}$, 
then the two copies of $K_r$ over the vertex sets
$\{x, v_1, \cdots, v_{r-1}\}, \{y, z_1, z_2, \cdots, z_{r-1}\}$ contain
at least $r+1$ overweight edges, while $R$ and $R'$ combined contain
at most $r$ overweight edges (see Fact \ref{fact:fact2}). Therefore
we conclude that there exists at most one overweight edge of
the form $\{y,z_i\}$.
\end{proof}

\begin{fact} \label{fact:fact4}
There does not exist a heavy $K_r$ over a vertex set of the form
$\{v_i, y_1, y_2, \cdots, y_{r-1}\}$ for $v_i \in L$ 
and $y_1, \cdots, y_{r-1} \in Y$. 
\end{fact}
\begin{proof}
Suppose that for some $v_i \in L$ and $y_1,\ldots,y_{r-1} \in Y$ the
vertices $\{v_1,y_1,\ldots,y_{r-1}\}$ induce a heavy $K_r$. 
Suppose that $\{y_1, \cdots, y_{r-1}\}$ are contained in $s$ disjoint copies 
$R_1, \cdots, R_s$ of $K_r$ belonging to 
$\mathcal{R}$, and let $x_1, \cdots, x_s$ be 
the `dominating' vertices of these $K_r$ guaranteed by Fact \ref{fact:fact2}
(note that $s \le r-1$).
If $s \le r-2$, then since each $x_i$ are incident to $L$ by at least $r-1$
overweight edges, we can find $s+1$ vertex-disjoint copies of a heavy $K_r$ over
the vertices $L \cup V(R_1) \cup \cdots \cup V(R_s)$. On the other hand, if $s = r-1$, then
there exists an index $j$ such that there exists $z \in V(R_j) \setminus \{x_j, y_1, \cdots, y_{r-1}\}$. Then by using the overweight edge $\{x_j, z\}$ 
(see Fact \ref{fact:fact2}) and the overweight edges 
between $\{x_1, \cdots, x_s\}$ and $L$, we can find
at least $s+1=r$ vertex-disjoint copies of a heavy $K_r$ over the
vertices $L \cup V(R_1) \cup \cdots \cup V(R_s)$. This contradicts the maximality of $\mathcal{R}$.
\end{proof}

For a set of vertices $T$, let $w(T) = \sum_{v_1,v_2 \in T} w(v_1,v_2)$. 
We will show
\[ \sum_{i = 1}^{r} \sum_{T \subset {Y \choose {r-1}}} w(\{v_i\} \cup T) > 
{r \choose 2}t r {|Y| \choose r-1}. \]
This inequality contradicts Fact \ref{fact:fact4}, showing that our original
assumption $\rho < \frac{n}{r}$ must be false.
Note that
\begin{align} \label{eq:trianglesum}
 &\sum_{i = 1}^{r} \sum_{T \subset {Y \choose {r-1}}} w(\{v_i\} \cup T) \nonumber 
  = {|Y|-1 \choose r-2} \sum_{i=1}^{r} \sum_{y \in Y} w(v_i, y)  + \frac{1}{2} r {|Y|-2 \choose r-3} \sum_{y_1 \in Y} \sum_{y_2 \in Y \setminus \{y_1\}} w(y_1,y_2)
\nonumber \\
  =& \, {|Y| \choose r-1} \left( 
	\frac{r-1}{|Y|} \sum_{i=1}^{r} \deg_w(v_i, Y)
	+ \frac{r(r-1)(r-2)}{2|Y|(|Y|-1)}\sum_{y \in Y} \deg_w(y, Y)
      \right) \nonumber \\
 \geq & \, {|Y| \choose r-1} \left( 
	\frac{r-1}{|Y|} \sum_{i=1}^{r} \deg_w(v_i, Y)
	+ \frac{r(r-1)(r-2)}{2|Y|^2}\sum_{y \in Y} \deg_w(y, Y)
      \right) 
\end{align}
where $\deg_w(v,Y)$ is the weighted degree of $v$ to vertices in $Y$.

For the first term on the right hand side of \eqref{eq:trianglesum}, we have
\begin{align}
  &\sum_{i=1}^{r} \deg_w(v_i, Y) \nonumber \\
  =&
  \sum_{i=1}^{r} \Big(\deg_w(v_i) - \deg_w(v_i, X) - \deg_w(v_i, V_R \setminus (X \cup Y)) - \deg_w(v_i, V \setminus V_R)) \Big) \nonumber \\
  \ge&
  \left(1 + (r-1)t + r\varepsilon \right)n - r \rho' - \left((r-2) + (r^2-r+2){r \choose 2}t\right) (\rho - \rho') - r{r \choose 2}t (n-r \rho). \nonumber
\end{align}
Since the coefficient of $n$ is positive for small enough $t$, we can
substitute $n > r\rho$ to get
\begin{equation}
  \sum_{i=1}^{r} \deg_w(v_i, Y) >
  r(r-1)t\rho + \left(2  - (r^2-r+2) {r \choose 2}t \right) (\rho - \rho') + r\varepsilon n.
  \label{eq:sum2}
\end{equation}

For the second term on the right hand side of \eqref{eq:trianglesum},
by Fact \ref{fact:fact3} and the fact $|Y|=(r-1)\rho'$, 
for a vertex $y \in Y$, we have
\begin{align*}
  \deg_w(y, Y)
  &\ge
  \left(\frac{1}{r} + \frac{r-1}{r}t + \varepsilon\right)n - \rho - {r \choose 2} t(n - \rho - |Y|) \\
  &=
  \left(\frac{1}{r} + \frac{r-1}{r}t - {r \choose 2}t + \varepsilon\right)n + {r \choose 2}t (r-1)\rho' - \left(1 - {r \choose 2}t\right) \rho.
\end{align*}
Since the coefficient of $n$ is positive for small enough $t$, we can
substitute $n > r\rho$ to get
\begin{equation} \label{eq:sum1}
  \deg_w(y, Y)  >  (r-1)t\rho - (r-1){r \choose 2}t(\rho - \rho')  + \varepsilon n.
\end{equation}

Using \eqref{eq:sum2} and \eqref{eq:sum1} in \eqref{eq:trianglesum},
\begin{align*}
& \frac{1}{{|Y| \choose r-1}} \sum_{i = 1}^{r} \sum_{T \subset {Y \choose {r-1}}} w(\{v_i\} \cup T) \\
  \ge& \frac{r-1}{|Y|} \Bigg( r(r-1)t\rho + \left(2  - (r^2-r+2) {r
      \choose 2}t \right) (\rho - \rho') + r\varepsilon n \Bigg)\\
&\quad + \frac{r(r-1)(r-2)}{2|Y|} \left( (r-1)t\rho - (r-1){r \choose
    2}t(\rho - \rho')  + \varepsilon n \right)
\end{align*}
If $t$ is small enough, then the coefficient of $\rho$ in the right hand side is
positive. Hence we can substitute $\rho \geq \rho'$ to get,
\begin{align*}
 \frac{1}{{|Y| \choose r-1}} \sum_{i = 1}^{r} \sum_{T \subset {Y
     \choose {r-1}}} w(\{v_i\} \cup T) \ge& \left(\frac{r(r-1)^2}{|Y|}
   + \frac{r(r-1)^2(r-2)}{2|Y|}\right)\left( t \rho' +
   \frac{\varepsilon n}{r-1} \right) \\
=& \frac{r^2(r-1)^2t\rho'}{2|Y|} + \frac{r^2(r-1)}{2|Y|}\varepsilon n \\
> & \frac{r^2(r-1)^2t\rho'}{2|Y|} \\
= & \frac{r(r-1)t}{|Y|}\binom{r}{2}\rho' \\
= & rt\binom{r}{2}.
\end{align*}
Since $|Y| = (r-1)\rho'$ and $\rho' \not= 0$ by Fact \ref{fact:fact15}, we have
\begin{equation*}
 \frac{1}{{|Y| \choose r-1}} \sum_{i = 1}^{r} \sum_{T \subset {Y
     \choose {r-1}}} w(\{v_i\} \cup T) > r\binom{r}{2} t.
\end{equation*}
contradicting Fact \ref{fact:fact4}. Thus our initial assumption
that $\rho < \frac{n}{r}$ must be false.   Hence  $\delta(r,t) \le \frac{1}{r} + \frac{r-1}{r}t + \varepsilon$ for every positive $\varepsilon$, and our claimed upper bound follows.
\end{proof}

It is worth noting that this proof gives a value of $t_r$ of
$\frac{4}{\binom{r}{2}\left(r^3-r^2 -2r + 4\right)}$. Thus
for $r=3$, we have $t_3 \ge \frac{1}{12}$.

For general values of $t$, we suggest 
two approaches which establish some upper bound (that unfortunately does
not match the lower bound given in Proposition \ref{P:lb}).

\subsection{First approach: hypergraphs}

In our first approach, we 
reduce our problem into the problem of finding a perfect matching
in hypergraphs as in Observation \ref{obs:simplecase}. 
The following lemma establishes the minimum number of
heavy $K_r$'s that each vertex must belong to in a given edge-weighted graph.

\begin{lemma} \label{lem:heavycliques}
If $w$ is a weight function with minimum weighted degree at least
$\delta n$, then every vertex is in at least 
$\left(1 - \frac{1-\delta}{1-t}\right) \binom{n-1}{r-1}$ cliques of size
$r \geq 3$ with weight at least $t\binom{r}{2}$.
\end{lemma}
\begin{proof}
Let $w$ be an arbitrary weight function, $v$ be an arbitrary
vertex, and let $\alpha_v$ be the number of $K_r$'s of weight
at least $t\binom{r}{2}$ containing $v$.  Now letting $S_v$ be the sum of the
weights of all $\binom{n-1}{r-1}$ $K_r$'s containing $v$, we have that 
\[ S_v \leq \alpha_v \binom{r}{2} + \paren{\binom{n-1}{r-1} -
  \alpha_v}t\binom{r}{2}.\] 
Let $W$ be the total weight of $w$. Since edges incident with $v$ occur
in $\binom{n-2}{r-2}$ $K_r$'s containing $v$ and the edges not incident to $v$ occur
in $\binom{n-3}{r-3}$ such $K_r$'s, we have
\[  S_v \geq \binom{n-3}{r-3}W + 
\paren{\binom{n-2}{r-2} - \binom{n-3}{r-3}} \deg_w(v). \] 
Combining these inequalities we have
\begin{align*}
\alpha_v &\geq \frac{1}{1-t}
\binom{n-1}{r-1} \paren{\frac{2(n-r)}{r(n-1)(n-2)} \deg_w(v) +
    \frac{2(r-2)}{r(n-1)(n-2)}W - t} \\
&\ge \frac{1}{1-t}
\binom{n-1}{r-1}\paren{\delta \frac{n}{n-1} - t} \\
&\ge \frac{\delta - t}{1-t}\binom{n-1}{r-1}. 
\end{align*}
\end{proof}

We  now apply Daykin and H\"{a}ggkvist's theorem \cite{Daykin:IndependentHypergraph} 
which asserts  that an $r$-uniform
hypergraph has a perfect matching if every vertex of it lies in at least
$\paren{1-\frac{1}{r}}\paren{\binom{n-1}{r-1}-1}$ hyperedges. This gives
the following bound:
\begin{proposition}
 For every $t \in (0,1]$ and $r \geq 3$ we have $\delta(r,t) \leq 1 - \frac{1-t}{r}$.
\end{proposition}

H\`{a}n, Person and Schacht \cite{Han:PerfectHypergraphMatchings} 
have conjectured that Daykin and H\"aggkvist's theorem can be improved, 
and an $r$-uniform hypergraph has a perfect matching if every vertex lies in
at least $(1-\paren{\frac{r-1}{r}}^{r-1} - \lilOh{1})\binom{n}{r-1}$
hyperedges.  If this conjecture were proved, then we would 
have $\delta(r,t) \leq 1 - (1-t)\paren{\frac{r-1}{r}}^{r-1}$. 

Since in \cite{Han:PerfectHypergraphMatchings} the
conjecture was proved  for $r=3$, we have that $\delta(3,t) \leq \frac{5}{9} +
\frac{4}{9}t$.  It is worth noting that this technique cannot be 
applied to obtain an upper bound matching  Proposition~\ref{P:lb}.
Consider the case when $r=3$ and $t = \frac{2}{3}$. 
The lower bound from Proposition \ref{P:lb}
reads as $\delta(3,\frac{2}{3}) \ge \frac{7}{9}$. To obtain 
a matching upper bound using this method, we would need to improve the conclusion
of Lemma \ref{lem:heavycliques} so that in every edge-weighted graph
of minimum degree at least $\frac{7}{9}n$, every vertex is contained
in at least $(\frac{5}{9}+o(1))\binom{n}{2}$ copies of $K_3$. However,
the following graph has minimum degree $\frac{29}{36}n$, and there are 
vertices which are contained in at most $\frac{319}{648}n^2$ copies of
$K_3$.  Let $A \cup B$ be a vertex partition such that $A$ has size
$\frac{29}{36}n$ and $B$ has size $\frac{7}{36}n$.  First, assign weight
$1$ to all the edges connecting $A$ and $B$ and give weight $1$ to an
$\frac{11}{18}n$-regular graph on $A$.  Give weight $0$ to each of the
remaining edges.  The minimum weighted degree of this graph is
$\frac{29}{36}n > \frac{7}{9}n > \frac{2}{3}n$, so this graph has a
triangle factor by the \HS  \ theorem.  However, every vertex in $B$ is only in
$\frac{29}{36}n \cdot \frac{11}{18
}n \cdot \frac{1}{2} =
\paren{\frac{319}{648} + \lilOh{1}}\binom{n}{2} < \paren{\frac{5}{9} +
  \lilOh{1}}\binom{n}{2}$ triangles. Similar constructions
can be made for other values of $r$ and $t$ as well.

\subsection{Second approach: induction}
\label{subsec:secondappraoch}
We  improve the upper bound by using  two reductive schemes to build a
$K_r$-factor out of a $K_{r'}$-factor of the graph (or a large portion of
the graph). 

{\bf Scheme 1.} Suppose $r = pq$ with $p, q > 1$, and let $w$ be an arbitrary weight
function with minimum weighted degree $\delta n$.  Let $\mathcal{K}$
be an arbitrary $K_p$-factor of $K_n$ with minimum average weight
$t_p$ and consider the weight function
$w_\mathcal{K}$ on $K_{n/p}$ defined as follows.
Associate to each vertex in $K_{n/p}$ a distinct clique in
$\mathcal{K}$; the weight of an edge is the average weight in $w$ of
the edges between the corresponding cliques.   Now the minimum
weighted degree under $w_\mathcal{K}$ is at least $\frac{p \delta n -
  p (p-1)}{p^2} = \paren{\delta - \frac{p-1}{n}}\frac{n}{p}$.  Letting
$K'$ be an arbitrary $K_q$-factor of this graph with minimum average
weight $t_q$, the factors $\mathcal{K}$ and $K'$ induce a $K_{pq}$-factor in
$K_n$ with minimum weight at least $t_q\binom{q}{2}p^2 +
t_p\binom{p}{2}q$.  Thus $\delta(pq,t,n) \leq \max\set{\delta(p,t,n),
  \delta(q,t,\frac{n}{p}) + \frac{p-1}{p}}$. Consequently, we
  have $\delta(pq,t) \le \max\set{\delta(p,t), \delta(q,t)}$.

{\bf Scheme 2.} Let $\delta' = \max\set{\delta(r-1,t), \frac{1}{2} + \frac{t}{2}}$.
We prove that $\delta(r,t) \le \delta'$. Let $\varepsilon$ be an arbitrary 
fixed positive real, and assume that $n_0$ is large enough so that 
$\delta(r-1,t,n) \le (\delta(r-1,t) + \frac{\varepsilon}{2})n$ for all $n \ge n_0$. Assume that
we are given an edge-weighted graph $G$ on $n \ge 2n_0$ vertices 
with minimum degree at least $(\delta'+\varepsilon)n$.
  We partition randomly the vertices of $G$
  into a set $A$ of size $\frac{r-1}{r}n$ and a set $B$ of size
  $\frac{1}{r} n = k$.  By the Chernoff-Hoeffding inequalities, for large 
  enough $n$, there is
  such a partition which additionally satisfies that for every
  vertex the weighted degree into $A$ is at least
  $(\delta'+\frac{\varepsilon}{2}) \frac{r-1}{r} n$ and into $B$ is at
  least $\delta' \frac{1}{r} n$. By the assumption on $\delta'$, 
  we can find a $K_{r-1}$-factor $\mathcal{K}_A$ on $A$ with 
  minimum average weight $t$. 
  
  Using $\mathcal{K}_A$ we construct a complete weighted bipartite
  graph $H$, where the vertices on one side are associated with
  cliques in $\mathcal{K}_A$ and the vertices on the other side are
  associated with vertices in $B$. For a clique $K \in \mathcal{K}_A$
  and a vertex $v \in B$, we assign as weight of the edge $(K, v)$,
  the average of the weights of the edges between $v$ and the vertices
  in $K$.  Notice that the minimum weighted degree of $H$ is at least
  $\delta'k \ge (\frac{1}{2} + \frac{t}{2})k$. Recall that 
  $H(t)$ is the unweighted subgraph of $H$ consisting
  of edges with weight at least $t$. By a similar argument as in
  Observation \ref{obs:simplecase}, the minimum degree in $H(t)$ is at least
  $\frac{k}{2}$. Thus by Hall's theorem, there is a  perfect
  matching $\mathcal{M}$ in $H(t)$.  
  
  Now notice that $\mathcal{K}_A$ and $\mathcal{M}$ lift to a
  $K_r$-factor of $G$ with minimum weight $t \binom{r-1}{2} + t
  (r-1) = t \binom{r}{2}$. Consequently, $\delta(r,t,n) \leq
  (\delta' + \varepsilon)n$. Since $\varepsilon$ can be arbitrarily small, 
  we have $\delta(r,t) \le \delta' = \max\set{\delta(r-1,t), \frac{1}{2} + \frac{t}{2}}$.



By Proposition \ref{P:lb}, Observation \ref{obs:simplecase}, and Scheme 2, 
we obtain the following theorem.
\begin{theorem} \label{thm:main}
For every $r \geq 3$ and $t \in (0,1]$, 
\[ \frac{1}{r} + \left( 1- \frac{1}{r} \right) t \leq \delta(r,t) \leq \frac{1}{2} +
\frac{t}{2} .\]
\end{theorem}
For the special case related to triangle factors that we discussed in the beginning, we have
$\frac{7}{9} \leq \delta(3,\frac{2}{3}) \leq \frac{5}{6}$.
We note that Theorem \ref{thm:main} has been proved without using Scheme~1,
however, Scheme~1 implies that if there is an improvement on the upper
bound for any $r$, then there is an improvement in the upper bound for
an infinite class of $r'$. For example, for any
fixed $k$, $\delta(r^k,t) \leq \delta(r,t)$.  Because of the
dependence on the bipartite matching result (which cannot be
improved) a similar statement does not hold using just Scheme~2.

\subsection{Open Question} 
In this article, we proposed the study of the function $\delta(r,t)$.
Based on the evidence given by Proposition \ref{P:lb} and Theorem \ref{thm:upper},
we make the following conjecture.

\begin{conj} 
For every $r \ge 2$ and $t \in (0,1]$, 
\[ \delta(r,t) = \frac{1}{r} + \left(1 - \frac{1}{r}\right)t. \]
\end{conj}

The function $\delta(r,t)$ shows different behavior 
from its non-weighted counterpart (which is related to the Hajnal-Szemer\'edi theorem).
As one can see from the discussion of Subsection \ref{subsec:secondappraoch},
the approach of examining the function by fixing $t$ and varying $r$,
opens up new possibilities which have no counterpart in
the Hajnal-Szemer\'edi theorem. 
We note that our results suggest, but do not quite establish, 
the fact that for fixed $t$, $\delta(r,t)$ is a \emph{decreasing} function of $r$. 
Further note that the weighted case has an extra power coming from the
ability to include any edge in a $K_r$-factor, even if that edge has weight $0$.  
This suggests that there could be a relation to results of 
Kuhn and Osthus~\cite{Kuhn:PerfectPacking} on the existence of $H$-factors in graphs.

\bigskip

\noindent \textbf{Acknowledgement.} We thank Wojciech Samotij
for his valuable comments.

\bibliographystyle{siam} 
\bibliography{HajnalSzemeredi}

\begin{thebibliography}{1}

\bibitem{Corradi:IndepCircuits}
{\sc K.~Corr\'adi and A.~Hajnal}, {\em On the maximal number of independent
  circuits in a graph}, Acta Math. Acad. Sci. Hungar., 14 (1963), pp.~423--439.

\bibitem{Daykin:IndependentHypergraph}
{\sc D.~E. Daykin and R.~H{\"a}ggkvist}, {\em Degrees giving independent edges
  in a hypergraph}, Bull. Austral. Math. Soc., 23 (1981), pp.~103--109.

\bibitem{Hajnal:EquitablePartition}
{\sc A.~Hajnal and E.~Szemer{\'e}di}, {\em Proof of a conjecture of {P}.
  {E}rd{\H o}s}, in Combinatorial theory and its applications, {II} ({P}roc.
  {C}olloq., {B}alatonf\"ured, 1969), North-Holland, Amsterdam, 1970,
  pp.~601--623.

\bibitem{Han:PerfectHypergraphMatchings}
{\sc H.~H{\`a}n, Y.~Person, and M.~Schacht}, {\em On perfect matchings in
  uniform hypergraphs with large minimum vertex degree}, SIAM J. Discrete
  Math., 23 (2009), pp.~732--748.

\bibitem{Kuhn:PerfectPacking}
{\sc D.~K{\"u}hn and D.~Osthus}, {\em The minimum degree threshold for perfect
  graph packings}, Combinatorica, 29 (2009), pp.~65--107.

\end{thebibliography}

\end{document}